\documentclass[12pt]{article}
\usepackage{amssymb,amsthm,amsmath,latexsym}
\usepackage{url}
\usepackage{color}
\newtheorem{thm}{Theorem}[section]
\newtheorem{prop}[thm]{Proposition}
\newtheorem{lem}[thm]{Lemma}

\theoremstyle{remark}
\newtheorem{rem}[thm]{Remark}
\newcommand{\FF}{\mathbb{F}}
\newcommand{\ZZ}{\mathbb{Z}}
\newcommand{\0}{\mathbf{0}}
\newcommand{\1}{\mathbf{1}}
\newcommand{\ww}{\omega}
\newcommand{\vv}{\omega^2}

\newcommand{\cC}{\mathcal{C}}
\newcommand{\cD}{\mathcal{D}}
\newcommand{\cP}{\mathcal{P}}

\DeclareMathOperator{\wt}{wt}
\DeclareMathOperator{\rank}{rank}

\begin{document}
\title{Quaternary  Hermitian linear complementary dual codes
}

\author{
Makoto Araya\thanks{Department of Computer Science,
Shizuoka University,
Hamamatsu 432--8011, Japan.
email: {\tt araya@inf.shizuoka.ac.jp}},
Masaaki Harada\thanks{
Research Center for Pure and Applied Mathematics,
Graduate School of Information Sciences,
Tohoku University, Sendai 980--8579, Japan.
email: {\tt mharada@tohoku.ac.jp}.}
and 
Ken Saito\thanks{
Research Center for Pure and Applied Mathematics,
Graduate School of Information Sciences,
Tohoku University, Sendai 980--8579, Japan.
email: {\tt kensaito@ims.is.tohoku.ac.jp}.}
}


\maketitle

\begin{abstract}
The largest minimum weights among 
quaternary Hermitian linear complementary dual codes
are known for dimension $2$.
In this paper, we give some conditions on the nonexistence of 
quaternary Hermitian linear complementary dual
codes with large minimum weights.
As a consequence, we completely determine the largest minimum
weights for dimension $3$, by using 
a classification of some quaternary codes.
In addition, for a positive integer $s$,
an entanglement-assisted quantum error-correcting
$[[21s+5,3,16s+3;21s+2]]$ code with maximal entanglement 
is constructed  for the first time
from a quaternary Hermitian linear complementary dual
$[26,3,19]$ code.
\end{abstract}

\section{Introduction}

Linear complementary dual (LCD for short) 
codes are linear codes that intersect with their dual
trivially.
LCD codes were introduced by Massey~\cite{Massey} and 
gave an optimum linear
coding solution for the two user binary adder channel.
Recently, much work has been done concerning LCD codes
for both theoretical and practical reasons
(see e.g.~\cite{AH}, \cite{CMTQ}, 
\cite{CMTQ2}, \cite{GOS}, \cite{LLG}, \cite{LLGF}
and the references given therein).
For example,
if there is a quaternary Hermitian LCD $[n,k,d]$ code,
then there is a maximal entanglement $[[n,k,d;n-k]]$ 
entanglement-assisted quantum error-correcting code 
(EAQECC for short) (see e.g.~\cite{LLG} and \cite{LLGF}).
From this point of view,
quaternary Hermitian LCD codes
play an important role in the study of
maximal entanglement EAQECC's.

It is a fundamental problem to determine the largest minimum
weight $d_4(n,k)$ among all quaternary Hermitian LCD 
$[n,k]$ codes for a given pair $(n,k)$.
It was shown that
$d_4(n,2)=\lfloor \frac{4n}{5} \rfloor$
if  $n \equiv 1,2,3 \pmod 5$ and
$d_4(n,2)=\lfloor \frac{4n}{5} \rfloor-1$
if  $n \equiv 0,4 \pmod 5$
for $n \ge 3$ in~\cite{Li} and \cite{LLGF}.
In this paper, we give some conditions on the nonexistence of 
quaternary Hermitian LCD codes with large minimum weights.
We give a classification of (unrestricted)
quaternary $[4r,3,3r]$ codes for $r=9,10,12,13,14,16$
and quaternary $[43,3,32]$ codes.
Using the above classification and the classification in~\cite{BGV}, 
we completely determine the largest minimum
weight among all quaternary Hermitian LCD 
codes of dimension $3$.
In addition, for a positive integer $s$,
a maximal entanglement $[[21s+5,3,16s+3;21s+2]]$ EAQECC
is constructed for the first time
from a quaternary Hermitian LCD $[26,3,19]$ code.

This paper is organized as follows.
In Section~\ref{Sec:2}, we prepare some definitions, notations
and basic results used in this paper. 
In Section~\ref{Sec:main}, we give characterizations of 
quaternary Hermitian LCD codes.
It is shown that there is no quaternary Hermitian LCD 
$[\frac{4^k-1}{3}s,k,4^{k-1}s]$ code
for $k \ge 3$ and $s \ge 1$ (Theorem~\ref{thm:S1}).
In addition, 
if $4 (4^{k-1}n-\frac{4^k-1}{3}\alpha) < k$, where $k \ge 3$ and $4\alpha-3n \ge 1$, then 
there is no quaternary Hermitian 
LCD $[n,k,\alpha]$ code $C$ with $d(C^{\perp_H}) \ge 2$, where
$d(C)$ denotes the minimum (Hamming)
weight of a quaternary code $C$ and $C^{\perp_H}$ denotes
the Hermitian dual code of $C$.
If $4(4^{k-1}n-\frac{4^k-1}{3}\alpha) \ge k \ge 3$, where $4\alpha-3n \ge 1$ and 
there is no quaternary Hermitian LCD 
$[4 (4^{k-1}n-\frac{4^k-1}{3}\alpha),k,
3(4^{k-1}n-\frac{4^k-1}{3}\alpha)]$ code $C_0$ with 
$d(C_0^{\perp_H}) \ge 2$, then 
there is no quaternary Hermitian LCD $[n,k,\alpha]$ code $C$ 
with $d(C^{\perp_H}) \ge 2$ (Theorem~\ref{thm:S-2-9}).
In Section~\ref{Sec:C}, from the
classification of quaternary codes of dimension $3$
by Bouyukliev, Grassl and Varbanov~\cite{BGV}, we 
determine $d_4(n,3)$ for $n \le 35$.
We emphasize that there is a quaternary Hermitian LCD 
$[26,3,19]$ code.  
This implies the existence of a quaternary Hermitian LCD 
$[21s+5,3,16s+3]$ code for $s \ge 1$
(Proposition~\ref{prop:1}).  
We also give a classification of 
quaternary $[4r,3,3r]$ codes for $r=9,10,12,13,14,16$
and quaternary $[43,3,32]$ codes.
In Section~\ref{Sec:dim3}, we completely determine $d_4(n,3)$ 
(Theorem~\ref{thm:main}).
This result is mainly obtained by applying
Theorems~\ref{thm:S1} and \ref{thm:S-2-9} to
the classification of some quaternary codes of dimension $3$
given in Section~\ref{Sec:C}.
As a consequence of Proposition~\ref{prop:1}, 
we show that there is
a maximal entanglement $[[21s+5,3,16s+3;21s+2]]$ EAQECC
for $s \ge 1$.  
This determines the largest minimum weight among
maximal entanglement $[[21s+5,3,d;21s+2]]$ EAQECC's as $16s+3$
(Remark~\ref{rem:Q}).
Finally, in Appendix, we give a proof of Proposition~\ref{prop:dim2-1}.

\section{Preliminaries}\label{Sec:2}

In this section, we prepare some definitions, notations
and basic results used in this paper.

\subsection{Definitions and notations}

We denote the finite field of order $4$
by $\FF_4=\{ 0,1,\ww , \vv  \}$, where $\omega^2 = \omega +1$.
For any element $\alpha \in \FF_{4}$, the conjugation of $\alpha$ is
defined as $\overline{\alpha}=\alpha^2$.
Throughout this paper, we use the following notations.
Let $\0_{s}$ and $\1_{s}$ denote the zero vector and the all-one vector of 
length $s$, respectively.
Let $O$ denote the zero matrix of appropriate size.
Let $I_k$ denote the identity matrix of order $k$ and
let $A^T$ denote the transpose of a matrix $A$.
For a matrix $A=(a_{ij})$, 
the conjugate matrix of $A$ is defined as
$\overline{A}=(\overline{a_{ij}})$.
For a $k \times n$ matrix $A$, we denote by $A^{(s)}$ 
the $k \times ns$ matrix
$
\left(
\begin{array}{cccccccc}
A & \cdots & A
\end{array}
\right).
$

A {\em quaternary} $[n,k]$ {\em code} $C$
is a $k$-dimensional vector subspace of $\FF_4^n$.
A generator matrix of a quaternary $[n,k]$ code $C$ is a $k \times n$
matrix such that the rows of the matrix generate $C$.
The {\em weight} $\wt(x)$ of a vector $x \in \FF_4^n$ is
the number of non-zero components of $x$.
A vector of $C$ is called a {\em codeword} of $C$.
The minimum non-zero weight of all codewords in $C$ is called
the {\em minimum weight} $d(C)$ of $C$. A quaternary $[n,k,d]$ code
is a quaternary $[n,k]$ code with minimum weight $d$.
The {\em weight enumerator} of a quaternary $[n,k]$ code $C$
is the polynomial $\sum_{i=0}^{n} A_i y^i$,
where $A_i$ denotes the number of codewords of weight $i$ in $C$.
Two quaternary $[n,k]$ codes $C$ and $C'$ are
{\em equivalent} 
if there is an $n \times n$ monomial matrix $P$ over $\FF_4$ with
$C' = \{ x P \mid x \in C\}$.

For any (unrestricted) quaternary $[n,k,d]$ code, 
the Griesmer bound is given by
\begin{equation}\label{eq:Gb}
n \ge \sum_{i=0}^{k-1} \left\lceil \frac{d}{4^i}\right\rceil.
\end{equation}
Throughout this paper, we use the following notation:
\begin{equation}\label{eq:Gb2}
\alpha_4(n,k)=\max\left\{d \in \mathbb{Z}_{\ge 0} \mid 
n \ge \sum_{i=0}^{k-1} \left\lceil \frac{d}{4^i}\right\rceil\right\},
\end{equation}
where $\ZZ_{\ge 0}$ denotes the set of nonnegative integers.

The {\em Euclidean dual} code $C^{\perp}$ of a quaternary $[n,k]$ code 
$C$ is defined as
$
C^{\perp}=
\{x \in \FF_4^n \mid \langle x,y\rangle = 0 \text{ for all } y \in C\},
$
where $\langle x,y\rangle = \sum_{i=1}^{n} x_i {y_i}$
for $x=(x_1,\ldots,x_n), y=(y_1,\ldots,y_n) \in \FF_4^n$.
The {\em Hermitian dual} code $C^{\perp_H}$ of a quaternary 
$[n,k]$ code $C$ is defined as
$
C^{\perp_H}=
\{x \in \FF_{4}^n \mid \langle x,y\rangle_H = 0 \text{ for all } y \in
C\
\},
$
where $\langle x,y\rangle_H= \sum_{i=1}^{n} x_i \overline{y_i}$
for $x=(x_1,\ldots,x_n), y=(y_1,\ldots,y_n) \in \FF_{4}^n$.
A quaternary code $C$ is called
{\em Euclidean linear complementary dual}
if $C \cap C^\perp = \{\0_n\}$.
A quaternary code $C$ is called
{\em Hermitian linear complementary dual}
if $C \cap C^{\perp_H} = \{\0_n\}$.
These two families of quaternary codes are collectively called
{\em linear complementary dual} (LCD for short) codes.
Note that quaternary Hermitian LCD
codes are also called {\em zero radical} codes
(see e.g.~\cite{LLG} and \cite{LLGF}).

A quaternary code $C$ is called
{\em Hermitian self-orthogonal}
if  $C \subset C^{\perp_H}$.
A quaternary code $C$ is called {\em even} if the weights
of all codewords of $C$ are even.
It is known that a quaternary code $C$ is 
Hermitian self-orthogonal
if and only if $C$ is even~\cite[Theorem~1]{MOSW}.
In addition, 
a quaternary code $C$ is Hermitian self-orthogonal
if and only if $G \overline{G}^T =O$
for a generator matrix $G$ of $C$.

A {\em $2$-$(v,k,\lambda)$ design} $\cD$ is a pair of a set $\cP$ of $v$ points
and a collection of $k$-element subsets of 
$\cP$ (called blocks)
such that every $2$-element subset of $\cP$ is contained in exactly
$\lambda$ blocks.
The number of blocks that contain a given
point is traditionally denoted by $r$, and the total number of
blocks is $b$.
Often a $2$-$(v,k,\lambda)$ design is simply called a $2$-design.
A $2$-design is called {\em symmetric} if $v=b$.
A $2$-design can be represented by its
{\em incidence matrix} $A=(a_{ij})$, where $a_{ij}=1$ if the $j$-th
point is contained in the $i$-th block
and $a_{ij}=0$ otherwise.

\subsection{Quaternary Hermitian LCD codes}

The following characterization gives a criterion for 
quaternary Hermitian LCD codes and is analogous to~\cite[Proposition 1]{Massey}.

\begin{prop}[{\cite[Proposition~3.5]{GOS}}]
Let $C$ be a quaternary code.  
Let $G$ be a generator matrix of $C$.
Then $C$ is Hermitian LCD if and only if
$G \overline{G}^T$ is nonsingular.
\end{prop}

Throughout this paper, we use the above characterization 
without mentioning this.

\begin{lem}\label{lem:dd1}
Suppose that there is a quaternary Hermitian LCD $[n,k,d]$ code $C$.
If $d_4(n-1,k) \le d-1$, then $d(C^{\perp_H}) \ge 2$.
\end{lem}
\begin{proof}
Suppose that $d(C^{\perp_H})=1$.
Then some column of a generator matrix of $C$ is $\0_k^T$.
By deleting this column, a  quaternary 
Hermitian LCD $[n-1,k,d]$ code is constructed.
This contradicts the assumption that $d_4(n-1,k) \le d-1$.
\end{proof}

\begin{lem}\label{lem:S-2-4}
Let $G_1$ and $G_2$ be generator matrices of a quaternary Hermitian LCD $[n_1,k,d_1]$ code and a quaternary Hermitian self-orthogonal $[n_2,k,d_2]$ code, respectively.
Then the code with generator matrix 
$
\left(
\begin{array}{cccccccc}
G_1 & G_2
\end{array}
\right)
$
is a quaternary Hermitian LCD $[n_1+n_2,k,d']$ code with $d' \ge d_1+d_2$.
\end{lem}
\begin{proof}
The straightforward proof is omitted.
\end{proof}

\subsection{Determination of $d_4(n,2)$}

Suppose that there is an (unrestricted)  quaternary 
$[n,2,d]$ code.
By the Griesmer bound~\eqref{eq:Gb},
we have
\begin{equation}\label{eq:d}
d \le
\left\lfloor \frac{4n}{5}\right\rfloor.
\end{equation}
Lu, Li, Guo and Fu~\cite[Lemma 3.1]{LLGF} constructed
a quaternary Hermitian LCD $[n,2,\lfloor \frac{4n}{5} \rfloor]$ code
for $n \equiv 1,2,3 \pmod 5$ and $n \ge 3$, and 
a quaternary Hermitian LCD $[n,2,\lfloor \frac{4n}{5} \rfloor-1]$ code
for $n \equiv 0,4 \pmod 5$ and $n \ge 4$.
The following proposition is mentioned in~\cite{LLGF}, 
by quoting~\cite{Li}.

\begin{prop}[Li~\cite{Li}]\label{prop:dim2-1}
If $n \equiv 0,4 \pmod 5$, then there is no quaternary Hermitian 
LCD $[n,2,\lfloor \frac{4n}{5} \rfloor]$
code.
\end{prop}

\begin{rem}
In Appendix, we give a proof of the above proposition
for the sake of completeness.
\end{rem}

Hence, one can determine $d_4(n,2)$ as follows:
\[
d_4(n,2)=
\begin{cases}
\lfloor \frac{4n}{5} \rfloor & \text{ if } n \equiv 1,2,3 \pmod 5,\\
\lfloor \frac{4n}{5} \rfloor -1 & \text{ if } n \equiv 0,4 \pmod 5,
\end{cases}
\]
for $n \ge 3$.


\section{Nonexistence of some quaternary  Hermitian LCD codes}
\label{Sec:main}

In this section, we give results on the nonexistence 
of some quaternary Hermitian LCD codes.

An easy counting argument yields the following lemma.
We give a proof for the sake of completeness.
Recall that  $b$ and $r$ denote the number of blocks of 
a $2$-$(v,k,\lambda)$ design $\cD$ and the number 
of blocks containing a given point of $\cD$, respectively.

\begin{lem}\label{lem:2des}
Let $n$ and $\alpha$ be positive integers.
Let $m=(m_1,\ldots,m_{v}) $ be a vector of  $\mathbb{Z}_{\ge 0}^{v}$.
Suppose that there is a $2$-$(v,k,\lambda)$ design $\cD$.
Let $A$ be the incidence matrix of $\cD$.
If each entry of the $b \times 1$ $\ZZ$-matrix $A m^T$ is 
at least $\alpha$,
then
\[
\frac{r\alpha-\lambda \sum_{j=1}^v m_j}
{r-\lambda} \le m_i \le \sum_{j=1}^v m_j -\frac{b-r}{r-\lambda}\alpha,
\]
for any $i \in \{1,\ldots,v\}$.
\end{lem}
\begin{proof}
Fix a point $p$ of $\cD$.
Define the following sets:
\[
X_0 =\{i \in \{1,\ldots,b\}\mid a_{ip}=0\} \text{ and }
X_1 =\{i \in \{1,\ldots,b\}\mid a_{ip}=1\},
\]
where $A=(a_{ij})$.
Let $w_i$ denote the $i$-th entry of the $b \times 1$ $\ZZ$-matrix $A m^T$. 
Then we have
\begin{align*}
\sum_{i \in X_1} w_i 
= r m_p + \sum_{j\in \{1,\ldots,v\}\setminus {\{p\}} } \lambda m_j
= (r-\lambda) m_p + \lambda \sum_{j=1}^v m_j.
\end{align*}
Since
$
\sum_{i \in X_0} w_i + \sum_{i \in X_1} w_i = \sum_{i=1}^b w_i
=r \sum_{j=1}^v m_j$,
we have
\[
\sum_{i \in X_0} w_i  
= (r-\lambda) \sum_{j=1}^v m_j - (r-\lambda) m_p.
\]
Since $|X_1|=r$ and $|X_0|=b-r$,
we have
\[
r \alpha \le \sum_{i \in X_1} w_i \text{ and }
(b-r) \alpha \le \sum_{i \in X_0} w_i.
\]
This completes the proof.
\end{proof}

According to~\cite{LLGF},
we define the $k \times (\frac{4^k-1}{3})$
matrices $S_{k}$
by inductive constructions as follows:
\begin{align*}
S_{1}&=
\begin{pmatrix}
1
\end{pmatrix}, \\
S_{k}&=
\begin{pmatrix}
S_{k-1} & \0_{\frac{4^{k-1}-1}{3}}^T & S_{k-1} & S_{k-1} & S_{k-1}\\
\0_{\frac{4^{k-1}-1}{3}} & 1
& \1_{\frac{4^{k-1}-1}{3}}  & \omega\1_{\frac{4^{k-1}-1}{3}}  
& \omega^2\1_{\frac{4^{k-1}-1}{3}} 
\end{pmatrix} \text{ if } k \ge 2.
\end{align*}
The matrix $S_{k}$ is a generator matrix of the quaternary 
$[\frac{4^k-1}{3},k,4^{k-1}]$ simplex code.
It is known that the quaternary $[\frac{4^k-1}{3},k,4^{k-1}]$ simplex 
code is a constant weight code.
More precisely, the code contains codewords of weights
$0$ and $4^{k-1}$ only.
Thus,
for $k \ge 2$, the quaternary  
$[\frac{4^k-1}{3},k,4^{k-1}]$ simplex code is even.
By~\cite[Theorem~1]{MOSW},
the quaternary $[\frac{4^k-1}{3},k,4^{k-1}]$ simplex code is
Hermitian self-orthogonal for $k \ge 2$.

Let $h_{k}^{(i)}$ be the $i$-th column of $S_{k}$.
For a vector $m=(m_1,\ldots,m_{\frac{4^k-1}{3}}) 
\in \mathbb{Z}_{\ge 0}^{\frac{4^k-1}{3}}$ with $\sum_{i} m_i=n$,
we define a matrix:
\[
\begin{array}{crcl}
G_{k}(m)=&
\left(
\begin{array}{cccccccc}
h_{k}^{(1)} \cdots h_{k}^{(1)} 
\end{array}
\right.
&
\cdots
&
\left.
\begin{array}{cccccccc}
h_{k}^{(\frac{4^k-1}{3})} \cdots h_{k}^{(\frac{4^k-1}{3})}
\end{array}
\right).\\
&
\underbrace{\hspace*{1.7cm}}_{m_1 \text{columns}}\quad
& &
\quad 
\underbrace{\hspace*{2.7cm}}_{m_{\frac{4^k-1}{3}} \text{columns}}
\end{array}
\]
For a quaternary $[n,k]$ code $C$ with $d(C^{\perp_H}) \ge 2$, there is a
vector $m=(m_1,\ldots,m_{\frac{4^k-1}{3}}
) \in \mathbb{Z}_{\ge 0}^{\frac{4^k-1}{3}}$ such that $C$ is equivalent to a code with generator matrix $G_{k}(m)$.
We denote the code by $C_{k}(m)$.

\begin{lem}\label{lem:S-2-6}
Suppose that $k \ge 3$,
$m=(
m_1,\ldots,m_{\frac{4^k-1}{3}}
) \in \mathbb{Z}_{\ge 0}^{\frac{4^k-1}{3}}$ and
$\sum_{i} m_i=n$.
If a quaternary $[n,k]$ code $C_{k}(m)$ has minimum weight at least
$\alpha$,
then
\begin{equation}\label{eq:S-2-6}
4\alpha-3n \le m_i \le n-\frac{4^{k-1}-1}{3\cdot 4^{k-2}}\alpha,
\end{equation}
for any $i \in \{1,\ldots,\frac{4^k-1}{3}\}$.
\end{lem}
\begin{proof}
It is well known that
the supports of the codewords of weight $4^{k-1}$ in the 
quaternary $[\frac{4^k-1}{3},k,4^{k-1}]$ simplex code form 
a symmetric $2$-$(\frac{4^k-1}{3},4^{k-1},3 \cdot 4^{k-2})$ design 
for $k \ge 3$ (see e.g.~\cite[p.~8]{CL}).
As the $2$-design $\cD$ in Lemma~\ref{lem:2des}, consider 
the symmetric $2$-$(\frac{4^k-1}{3},4^{k-1},3 \cdot 4^{k-2})$ design.
Since $d(C_{k}(m)) \ge \alpha$,
it follows from the structure of $C_{k}(m)$ that
each entry of the $\frac{4^k-1}{3} \times 1$ matrix 
$A m^T$ is at least $\alpha$.
The result follows from  Lemma~\ref{lem:2des}.
\end{proof}

If we write $n=\frac{4^k-1}{3}s$, then
$\alpha_4(n,k)=4^{k-1}s$ (see~\eqref{eq:Gb2} for $\alpha_4(n,k)$).

\begin{thm}\label{thm:S1}
Suppose that $k \ge 3$.
If $n \equiv 0 \pmod{\frac{4^k-1}{3}}$, then there is no quaternary Hermitian LCD $[n,k,\alpha_4(n,k)]$ code.
\end{thm}
\begin{proof}
Write $n=\frac{4^k-1}{3}s$.
Suppose that there is a quaternary Hermitian 
LCD $[\frac{4^k-1}{3}s,k,4^{k-1}s]$ code $C$.
Since $\alpha_4(\frac{4^k-1}{3}s-1,k) \le 4^{k-1}s-1$, 
$d(C^{\perp_H}) \ge 2$ by Lemma~\ref{lem:dd1}.
Thus, we may assume that
$C$ is equivalent to a code $C_{k}(m)$ for some vector 
$m=(m_1,\ldots,m_{\frac{4^k-1}{3}}) \in \mathbb{Z}_{\ge 0}^{\frac{4^k-1}{3}}$.
Consider the conditions given in~\eqref{eq:S-2-6}.
Since we have
\begin{align*}
4^{k}s -3 \left(\frac{4^k-1}{3}s\right)=s \text{ and }
\frac{4^k-1}{3}s-\frac{4^{k-1}-1}{3\cdot 4^{k-2}}4^{k-1}s=s,
\end{align*}
by Lemma~\ref{lem:S-2-6}, we have $m_i=s$ for each $i$.
This means that $m=s\1_{\frac{4^k-1}{3}}$.
Since $S_{k}\overline{S_{k}}^T=O$, we have
\[G_{k}(s\1_{\frac{4^k-1}{3}})\overline{G_{k}
(s\1_{\frac{4^k-1}{3}})}^T
=S_{k}^{(s)}\overline{{S_{k}^{(s)}}}^T=O.\]
By~\cite[Theorem~1]{MOSW}, $C_{k}(m)$ is quaternary 
Hermitian self-orthogonal, which
is a contradiction.
\end{proof}

Set
\[r_4(n,k,\alpha)=4^{k-1}n-\frac{4^k-1}{3}\alpha,\]
for positive integers $n,k$ and $\alpha$.
The following theorem is one of the main results in this paper.

\begin{thm}\label{thm:S-2-9}
Suppose that $k \ge 3$ and $4\alpha-3n \ge 1$.
\begin{itemize}
\item[{\rm (i)}] Suppose that $4r_4(n,k,\alpha) < k$.
Then there is no quaternary Hermitian 
LCD $[n,k,\alpha]$ code $C$ with $d(C^{\perp_H}) \ge 2$.
\item[{\rm (ii)}] 
Suppose that $4r_4(n,k,\alpha) \ge k$.
If there is no quaternary Hermitian LCD $[4r_4(n,k,\alpha),k,3r_4(n,k,\alpha)]$ code $C_0$ with $d(C_0^{\perp_H}) \ge 2$, then there is no quaternary Hermitian LCD $[n,k,\alpha]$ code $C$ with $d(C^{\perp_H}) \ge 2$.
\end{itemize}
\end{thm}
\begin{proof}
Suppose that there is a quaternary Hermitian LCD $[n,k,\alpha]$ code $C$ with $d(C^{\perp_H}) \ge 2$.
Then $C$ is equivalent to a code $C_{k}(m)$ with generator matrix $G_{k}(m)$
for some vector $m=(m_1,\ldots,m_{\frac{4^k-1}{3}}) \in \mathbb{Z}_{\ge 0}^{\frac{4^k-1}{3}}$.
Since $d(C)=\alpha$, by Lemma~\ref{lem:S-2-6}, we have
\[4\alpha-3n \le m_i,\]
for each $i \in\{1,\ldots,\frac{4^k-1}{3}\}$.
Thus, at least $4\alpha-3n$ columns of the matrix $G_{k}(m)$ are $h_{k}^{(i)}$,
then we obtain a matrix $G$ of the following form:
\[G=
\left(
\begin{array}{cccccccc}
G_0 & S_{k}^{(4\alpha-3n)}
\end{array}
\right),
\]
by permuting columns of $G_{k}(m)$.
Here, $G_0$ is a $k \times n_0$ matrix, where
\[
n_0=4\left(4^{k-1}n-\frac{4^k-1}{3}\alpha\right)=4r_4(n,k,\alpha).
\]
The code $C'$ with generator matrix $S_{k}^{(4\alpha-3n)}$ is a 
quaternary Hermitian self-orthogonal $[n',k,d']$ code, where
\begin{align*}
n'=(4\alpha-3n) \frac{4^k-1}{3} \text{ and }
d'=(4\alpha-3n)4^{k-1}.
\end{align*}
Since $S_{k}\overline{S_{k}}^T=O$, we have 
$G\overline{G}^T=G_0\overline{G_0}^T$.
Since $\rank(G\overline{G}^T)=k$, we have
\[\rank(G_0) \ge \rank(G_0\overline{G_0}^T)=k.\]

\begin{itemize}
\item[{\rm (i)}] 
Suppose that $4r_4(n,k,\alpha) < k$.
Since $G_0$ is a $k \times 4r_4(n,k,\alpha)$ matrix,
\[
k > 4r_4(n,k,\alpha) \ge \rank(G_0) \ge k,
\]
which is a contradiction.

\item[{\rm (ii)}] Suppose that $4r_4(n,k,\alpha) \ge k$.
Let $C_0$ be the quaternary code with generator matrix $G_0$.
Since $G\overline{G}^T=G_0\overline{G_0}^T$, 
$C_0$ is a quaternary Hermitian LCD $[n_0,k]$ code.
It follows from the form of $G_0$ that $d(C_0^{\perp_H}) \ge 2$.
Let $d_0$ denote the minimum weight of $C_0$.
By Lemma~\ref{lem:S-2-4},  $\alpha \ge d_0+d'$.
Since $C'$ is a constant weight code, there is a codeword of weight $d_0+d'$ in $C_{k}(m)$.
Thus, $\alpha=d_0+d'$ then we have
\[d_0=
-(4^k-1)\alpha +4^{k-1} 3 n
=3r_4(n,k,\alpha).
\]
Therefore, 
there is a quaternary Hermitian LCD $[4r_4(n,k,\alpha),k,3r_4(n,k,\alpha)]$ 
code $C_0$ with $d(C_0^{\perp_H}) \ge 2$.
\end{itemize}
This completes the proof. 
\end{proof}

\begin{rem}
If $4\alpha-3n \ge 1$, then we have
\[n \ge 4r_4(n,k,\alpha)+\frac{4^k-1}{3},\]
since
$
n-4r_4(n,k,\alpha)
\ge \frac{1}{3}(-3n(4^k-1)+(3n+1)(4^k-1)) 
=\frac{1}{3}(4^k-1).
$
\end{rem}

\section{Quaternary codes of dimension 3}
\label{Sec:C}

In this section, a classification of (unrestricted) quaternary codes
of dimension $3$ is done for some lengths by using 
computer calculations (see Lemma~\ref{lem:S3} for the motivation of
our classification).
All computer calculations were
done by programs in {\sc Magma}~\cite{Magma} and 
programs in the language C.

\subsection{Classification method}\label{Sec:M}

A {\em shortened code} $C'$ of a  quaternary code $C$ 
is the set of all codewords
in $C$ which are $0$ in a fixed coordinate with that
coordinate deleted.
A shortened code $C'$ of a  quaternary $[n,k,d]$ code $C$ with $d \ge 2$
is a  quaternary $[n-1,k,d]$ code if the deleted coordinate
is a zero coordinate and a  quaternary $[n-1,k-1,d']$
code with $d' \ge d$ 
otherwise.

By considering the inverse operation of shortening,
every quaternary $[n,3,d]$ code with $d \ge 2$ is constructed from some
quaternary $[n-1,2,d']$ code with $d' \ge d$.
By considering equivalent  quaternary codes, we may assume that
a quaternary $[n-1,2,d']$ code has the following generator
matrix:
\begin{equation}\label{eq:GM}
\left(\begin{array}{cc|cc|cccccccc}
1&0& \0_{a_1}&\0_{a_2}&\1_{a_3}&\1_{a_4}&\1_{a_5}&\1_{a_6} \\
0&1& \0_{a_1}&\1_{a_2}&\0_{a_3}&\1_{a_4}&\ww \1_{a_5}&\vv \1_{a_6}\\
\end{array}\right),
\end{equation}
where
$a_1+a_2=n-d'-2$ and $a_3+a_4+a_5+a_6=d'-1$.
For the generator matrix~\eqref{eq:GM} of
each of all inequivalent quaternary $[n-1,2,d']$ codes with $d' \ge d$, 
consider the generator matrices
$
\left(\begin{array}{cc}
I_3 & M 
\end{array}\right),
$
where
\begin{equation}\label{eq:GM3}
M=
\left(\begin{array}{cc|cccccccc}
 \0_{a_1}&\0_{a_2}&\1_{a_3}&\1_{a_4}&\1_{a_5}&\1_{a_6} \\
 \0_{a_1}&\1_{a_2}&\0_{a_3}&\1_{a_4}&\ww \1_{a_5}&\vv \1_{a_6}\\
 x_1 &x_2 &x_3 &x_4 &x_5& x_6 \\
\end{array}\right),
\end{equation}
where $x_i=(x_{i,1},\ldots,x_{i,a_i})$,
under the condition that $x_{i,j} \le x_{i,k}$ for $j < k$
and $x_{1,\ell} \in \{0,1\}$.
Here, we consider a natural order on the elements of $\FF_4$
as follows $0 <1 <\omega < \omega^2$.
In this way, all quaternary $[n,3,d]$ codes,
which must be checked further for equivalences,
are constructed.
By checking equivalences among these codes,  
we complete a classification of quaternary $[n,3,d]$ codes.

\subsection{Lengths up to 35}
\label{Sec:35}

Here we investigate the values $d_4(n,3)$ for $n=4,5,\ldots,35$.
Let $d_4^{\text{all}}(n,3)$ denote the largest minimum
weight among all (unrestricted) quaternary $[n,3]$ codes
(see~\cite{G} for the current information on $d_4^{\text{all}}(n,3)$). 
Lu, Li, Guo and Fu~\cite{LLGF} found quaternary 
Hermitian LCD codes having large minimum weights for dimension $3$.
From~\cite[Tables 3 and 4]{LLGF}, we know
$d_4(n,3)=d_4^{\text{all}}(n,3)$
for
\[
n=7,8,9,10,11,12,13,17,18,23,24,25,28,29,30,33,34.
\]
For $n \le 35$,
Bouyukliev, Grassl and Varbanov~\cite{BGV} completed
the classification of (unrestricted)
quaternary $[n,3,d_4^{\text{all}}(n,3)]$ codes.
The number of
the inequivalent quaternary $[n,3,d_4^{\text{all}}(n,3)]$ codes 
are given in~\cite[Table~3]{BGV}.
Based on the number given in~\cite[Table~3]{BGV}, 
we reconstructed all inequivalent
quaternary $[n,3,d_4^{\text{all}}(n,3)]$ codes for
\[
n=
4,5,6,14,15,16,19,20,21,22,26,27,31,32,35,
\]
by using the method in Section~\ref{Sec:M}.
Then we found that $d_4(n,3) < d_4^{\text{all}}(n,3)$
for the above lengths except $26$.
For length $26$, 
we found that $d_4(n,3) = d_4^{\text{all}}(n,3)$.
For the remaining lengths, 
from~\cite[Tables 3 and 4]{LLGF}, we know
$d_4(n,3) = d_4^{\text{all}}(n,3)-1$.
This determines the
largest minimum weight $d_4(n,3)$ for lengths $n=4,5,\ldots,35$,
where the results are listed in Table~\ref{Tab:35}.
In the table, the reference about the existence of  quaternary 
Hermitian LCD $[n,3,d_4(n,3)]$ codes is also listed.

\begin{table}[thb]
\caption{$d_4(n,3)$ $(n=4,5,\ldots,35)$}
\label{Tab:35}
\begin{center}
{\small
\begin{tabular}{c|c|c|c|c|c}
\noalign{\hrule height0.8pt}
$n$ & $d_4(n,3)$ & Reference &
$n$ & $d_4(n,3)$ & Reference \\
\hline
 4 &  1&~\cite[Table 3]{LLGF} &20 & 14&~\cite[Table 3]{LLGF} \\
 5 &  2&~\cite[Table 3]{LLGF} &21 & 15&~\cite[Table 3]{LLGF} \\
 6 &  3&~\cite[Table 3]{LLGF} &22 & 15&~\cite[Table 4]{LLGF} \\
 7 &  4&~\cite[Table 3]{LLGF} &23 & 16&~\cite[Table 4]{LLGF} \\
 8 &  5&~\cite[Table 3]{LLGF} &24 & 17&~\cite[Table 4]{LLGF} \\
 9 &  6&~\cite[Table 3]{LLGF} &25 & 18&~\cite[Table 4]{LLGF} \\
10 &  6&~\cite[Table 3]{LLGF} &26 & 19& $C_{26}$             \\
11 &  7&~\cite[Table 3]{LLGF} &27 & 19&~\cite[Table 4]{LLGF} \\
12 &  8&~\cite[Table 3]{LLGF} &28 & 20&~\cite[Table 4]{LLGF} \\
13 &  9&~\cite[Table 3]{LLGF} &29 & 21&~\cite[Table 4]{LLGF} \\
14 &  9&~\cite[Table 3]{LLGF} &30 & 22&~\cite[Table 4]{LLGF} \\
15 & 10&~\cite[Table 3]{LLGF} &31 & 22&~\cite[Table 4]{LLGF} \\
16 & 11&~\cite[Table 3]{LLGF} &32 & 23&~\cite[Table 4]{LLGF} \\
17 & 12&~\cite[Table 3]{LLGF} &33 & 24&~\cite[Table 4]{LLGF} \\
18 & 13&~\cite[Table 3]{LLGF} &34 & 25&~\cite[Table 4]{LLGF} \\
19 & 13&~\cite[Table 3]{LLGF} &35 & 25&~\cite[Table 4]{LLGF} \\
\noalign{\hrule height0.8pt}
\end{tabular}
}
\end{center}
\end{table}

We give details for the case $d_4(26,3)=19$.
There are five inequivalent  quaternary 
$[26,3,19]$ codes~\cite[Table~3]{BGV}.
We verified that one of them is Hermitian LCD.  
This code $C_{26}$ has the following
generator matrix:
\begin{multline*}
\left(\begin{array}{ccccccccccccccccccccccccccccccc}
1&0&0&0&0&0&0&0&1&1&1&1&1\\
0&1&0&0&1&1&  1&  1&0&0&  0&   0&1\\
0&0&1&1&1&1&\ww&\ww&1&1&\ww&\ww^2&0\\
\end{array}\right. 
\\
\left.
\begin{array}{ccccccccccccccccccccccccccccccc}
1&1&1&1&1&1&1&1&1&1&1&1&1\\
1&1&1&1&\ww&\ww&\ww&\ww&\ww^2&\ww^2&\ww^2&\ww^2&\ww^2\\
1&\ww&\ww&\ww^2&0&1&\ww&\ww^2&0&1&\ww&\ww^2&\ww^2 \\
\end{array}\right), 
\end{multline*}
and the following weight enumerator:
\[
1+ 33 y^{19}
+ 18 y^{20}
+  3 y^{21}
+  9 y^{22}.
\]

\subsection{Lengths 36, 40, 43, 48, 52, 56 and 64}
\label{Sec:36-64}

By using the method in Section~\ref{Sec:M},
a classification of (unrestricted)
quaternary $[4r,3,3r]$ codes for $r=9,10,12,13,14,16$
and quaternary $[43,3,32]$ codes
was done.  
These codes have minimum weights $d_4^{\text{all}}(n,3)$.
To save space, the results are given only.

\begin{prop}\label{prop:36-64}
\begin{itemize}
\item[\rm (i)]
There are two inequivalent quaternary $[4r,3,3r]$ codes 
$C_{4r,i}$ $(i=1,2)$, none of
which is Hermitian LCD for $r=9,10$.

\item[\rm (ii)]
There are ten inequivalent quaternary $[43,3,32]$ codes
$C_{43,i}$ $(i=1,2,\ldots,10)$, none of
which is Hermitian LCD.

\item[\rm (iii)]
There are five inequivalent quaternary $[4r,3,3r]$ codes
$C_{4r,i}$ $(i=1,2,\ldots,5)$, none of
which is Hermitian LCD for $r=12,13$.

\item[\rm (iv)]
There are six inequivalent quaternary $[56,3,42]$ codes
$C_{56,i}$ $(i=1,2,\ldots,6)$, none of
which is Hermitian LCD.

\item[\rm (v)]
There are $15$ inequivalent quaternary $[64,3,48]$ codes
$C_{64,i}$ $(i=1,2,\ldots,15)$, none of
which is Hermitian LCD.
\end{itemize}
\end{prop}

\begin{table}[thbp]
\caption{Vectors $v_j^T$ $(j=1,2,\ldots,22)$}
\label{Tab:vec}
\begin{center}
{\footnotesize
\begin{tabular}{c|c|c|c|c|c|c|c|c|c}
\noalign{\hrule height0.8pt}
$i$ & $v_i^t$ &
$i$ & $v_i^T$ &
$i$ & $v_i^T$ &
$i$ & $v_i^T$ &
$i$ & $v_i^T$ \\
\hline
${1}$ &$( 0,0,1 )$&
 ${6}$ &$( 0,1,\vv )$&
  ${11}$&$( 1,1,0 )$&
   ${16}$&$( 1,\ww,1 )$&
    ${21}$&$( 1,\vv,\ww )$
\\
${2}$ &$( 0,0,1 )$&
 ${7}$ &$( 1,0,0 )$&
  ${12}$&$( 1,1,1 )$&
   ${17}$&$( 1,\ww,\ww )$&
     ${22}$&$( 1,\vv,\vv )$
\\
${3}$ &$( 0,1,0 )$&
 ${8}$ &$( 1,0,1 )$&
  ${13}$&$( 1,1,\ww )$&
   ${18}$&$( 1,\ww,\vv )$&\\
${4}$ &$( 0,1,1 )$&
 ${9}$ &$( 1,0,\ww )$&
  ${14}$&$( 1,1,\vv )$&
   ${19}$&$( 1,\vv,0 )$&\\
${5}$ &$( 0,1,\ww )$&
 ${10}$&$( 1,0,\vv )$&
  ${15}$&$( 1,\ww,0 )$&
   ${20}$&$( 1,\vv,1 )$&\\
\noalign{\hrule height0.8pt}
\end{tabular}
}
\end{center}
\end{table}

In order to display the matrices 
$M$ in~\eqref{eq:GM3} for generator matrices 
$
\left(\begin{array}{cc}
I_3 & M 
\end{array}\right)
$
of $C_{4r,i}$ ($r=9,10,12,13,14,16$)
and $C_{43,i}$,
we give some vectors $v_i^T$ of length $3$
in Table~\ref{Tab:vec}.
Let $n_j$ $(j=1,2,\ldots,22)$ be the number of the columns of 
$M$ in~\eqref{eq:GM3}, which are equal to $v_j$.
The numbers $n_j$ $(j=1,2,\ldots,22)$ are
listed in Tables~\ref{Tab:36-64-1} and \ref{Tab:36-64-2}.
The weight enumerators $W_{n,i}$ of $C_{n,i}$
are listed in Table~\ref{Tab:W}.

\begin{table}[thb]
\caption{$(n_1,n_2,\ldots,n_{22})$}
\label{Tab:36-64-1}
\begin{center}
{\footnotesize
\begin{tabular}{c|c}
\noalign{\hrule height0.8pt}
Code & $(n_1,n_2,\ldots,n_{22})$ \\
\hline
$C_{36,1}$&$(0, 1, 0, 2, 2, 2, 0, 2, 2, 2, 1, 2, 2, 2, 0, 2, 2, 2, 1, 2, 2, 2)$\\
$C_{36,2}$&$(0, 1, 0, 2, 2, 2, 0, 2, 2, 2, 1, 2, 2, 2, 1, 1, 2, 2, 1, 2, 2, 2)$\\
\hline
$C_{40,1}$& $(0, 0, 1, 2, 2, 2, 1, 2, 2, 2, 2, 2, 2, 2, 2, 1, 2, 2, 2, 2, 2, 2)$\\
$C_{40,2}$& $(0, 1, 1, 2, 2, 2, 1, 2, 0, 2, 2, 2, 2, 2, 2, 2, 2, 2, 2, 2, 2, 2)$\\
\noalign{\hrule height0.8pt}
\end{tabular}
}
\end{center}
\end{table}

\begin{table}[thbp]
\caption{$(n_1,n_2,\ldots,n_{22})$}
\label{Tab:36-64-2}
\begin{center}
{\footnotesize
\begin{tabular}{c|c}
\noalign{\hrule height0.8pt}
Code & $(n_1,n_2,\ldots,n_{22})$ \\
\hline
$C_{43,1}$&$(0,0,1,2,3,3,0,1,1,1,2,3,3,2,2,3,3,2,2,2,2,2)$\\
$C_{43,2}$&$(0,0,0,3,3,3,1,2,2,2,2,2,2,2,2,2,2,2,2,2,2,2)$\\
$C_{43,3}$&$(0,1,1,2,2,3,1,0,2,2,2,2,3,2,3,2,2,2,2,2,2,2)$\\
$C_{43,4}$&$(0,0,1,2,3,3,1,2,2,2,2,2,2,2,2,2,2,2,2,2,2,2)$\\
$C_{43,5}$&$(0,0,1,2,3,3,0,2,1,2,2,2,2,2,2,3,2,3,2,2,2,2)$\\
$C_{43,6}$&$(0,1,1,2,2,3,1,2,2,2,2,2,2,2,2,2,2,2,2,2,2,2)$\\
$C_{43,7}$&$(0,0,0,3,3,3,1,0,2,2,3,1,3,3,3,1,3,3,2,0,2,2)$\\
$C_{43,8}$&$(0,0,0,3,3,3,0,1,2,2,2,2,3,3,2,2,3,3,1,1,2,2)$\\
$C_{43,9}$&$(0,0,1,2,3,3,0,2,2,1,2,3,2,3,2,2,3,3,0,2,2,2)$\\
$C_{43,10}$&$(1,1,1,2,2,2,1,2,2,2,2,2,2,2,2,2,2,2,2,2,2,2)$\\
\hline
$C_{48,1}$&$(0,0,1,3,3,3,1,3,3,3,0,1,1,1,2,3,3,3,2,3,3,3)$\\
$C_{48,2}$&$(0,0,1,3,3,3,0,2,2,2,1,2,2,2,2,3,3,3,2,3,3,3)$\\
$C_{48,3}$&$(0,1,1,2,3,3,1,3,2,3,2,3,3,2,0,2,2,2,2,2,3,3)$\\
$C_{48,4}$&$(0,1,1,2,3,3,1,3,2,3,2,2,2,2,2,3,2,3,2,2,2,2)$\\
$C_{48,5}$&$(0,2,2,0,3,3,2,0,3,3,0,3,3,3,3,3,3,0,3,3,0,3)$\\
\hline
$C_{52,1}$&$(0,0,2,3,3,3,1,3,3,3,2,2,2,2,3,3,3,3,2,2,2,2)$\\
$C_{52,2}$&$(0,0,1,3,3,3,0,2,2,2,2,3,3,3,2,3,3,3,2,3,3,3)$\\
$C_{52,3}$&$(0,0,2,3,3,3,0,0,1,1,3,3,3,3,3,3,3,3,3,3,3,3)$\\
$C_{52,4}$&$(0,0,2,3,3,3,1,3,3,3,1,1,1,1,3,3,3,3,3,3,3,3)$\\
$C_{52,5}$&$(0,1,1,3,3,3,1,3,3,3,0,2,2,2,2,3,3,3,2,3,3,3)$\\
\hline
$C_{56,1}$&$(0,0,2,3,3,3,2,3,3,3,1,2,2,2,3,3,3,3,3,3,3,3)$\\
$C_{56,2}$&$(0,0,1,3,3,3,1,3,3,3,2,3,3,3,2,3,3,3,2,3,3,3)$\\
$C_{56,3}$&$(0,1,2,3,3,3,1,0,2,2,3,3,3,3,3,3,3,3,3,3,3,3)$\\
$C_{56,4}$&$(0,1,2,3,3,3,1,2,3,3,2,3,2,3,2,3,3,2,3,3,3,3)$\\
$C_{56,5}$&$(0,0,2,3,3,3,1,2,2,2,3,3,3,3,3,3,3,3,2,3,3,3)$\\
$C_{56,6}$&$(0,1,1,3,3,3,1,2,3,3,2,3,3,3,2,3,3,3,2,3,3,3)$\\
\hline
$C_{64,1}$&$(0,0,2,4,4,4,1,3,3,3,3,4,4,4,2,3,3,3,2,3,3,3)$\\
$C_{64,2}$&$(0,1,1,4,4,4,1,2,3,3,2,2,3,3,3,3,4,4,3,3,4,4)$\\
$C_{64,3}$&$(0,2,2,3,3,4,2,3,3,4,0,3,3,3,3,3,4,3,3,4,3,3)$\\
$C_{64,4}$&$(0,2,2,3,3,4,2,3,3,3,3,3,3,3,3,3,3,3,3,3,3,3)$\\
$C_{64,5}$&$(0,0,2,4,4,4,0,0,1,1,4,3,4,4,4,3,4,4,4,3,4,4)$\\
$C_{64,6}$&$(0,0,2,4,4,4,2,4,4,4,2,3,3,3,3,4,4,4,1,2,2,2)$\\
$C_{64,7}$&$(0,1,2,3,4,4,2,4,4,3,3,4,3,4,3,3,4,4,0,2,2,2)$\\
$C_{64,8}$&$(0,1,2,3,4,4,2,3,3,3,3,3,3,3,3,3,3,3,3,3,3,3)$\\
$C_{64,9}$&$(0,0,2,4,4,4,2,3,3,3,3,3,3,3,3,3,3,3,3,3,3,3)$\\
$C_{64,10}$&$(0,1,2,3,4,4,1,2,2,2,3,4,4,3,3,4,4,3,3,3,3,3)$\\
$C_{64,11}$&$(0,1,2,3,4,4,2,3,3,3,3,4,4,3,2,3,3,2,3,3,3,3)$\\
$C_{64,12}$&$(0,1,1,4,4,4,1,4,4,4,2,4,4,4,2,4,4,4,0,2,2,2)$\\
$C_{64,13}$&$(1,2,2,3,3,3,2,3,3,3,3,3,3,3,3,3,3,3,3,3,3,3)$\\
$C_{64,14}$&$(0,1,1,4,4,4,2,3,3,3,3,3,3,3,3,3,3,3,3,3,3,3)$\\
$C_{64,15}$&$(0,3,3,0,4,4,3,0,4,4,0,4,4,4,4,4,4,0,4,4,0,4)$\\
\noalign{\hrule height0.8pt}
\end{tabular}
}
\end{center}
\end{table}

\begin{table}[thb]
\caption{Weight enumerators}
\label{Tab:W}
\begin{center}
{\footnotesize
\begin{tabular}{c|l|c|l}
\noalign{\hrule height0.8pt}
$W_{n,i}$ &\multicolumn{1}{c|}{Weight enumerator} &
$W_{n,i}$ &\multicolumn{1}{c}{Weight enumerator} \\
\hline
$W_{36,1}$& $1+ 48y^{27}+  12y^{28} + 3 y^{32}$ &
$W_{36,2}$& $1+ 45y^{27}+  15y^{28} + 3 y^{31}$ \\
\hline
$W_{40,1}$&$1+ 36 y^{30} +24 y^{31} +3 y^{32}$&
$W_{40,2}$&$1+ 48 y^{30}+ 15y^{32}$ \\
\hline
$W_{43,1}$&$1+ 45 y^{32} + 15 y^{34} + 3y^{38}$&
$W_{43,2}$&$1+ 39 y^{32} + 24 y^{34}$\\
$W_{43,3}$&$1+ 39 y^{32} + 12 y^{33} + 12y^{35}$&
$W_{43,4}$&$1+ 27 y^{32} + 24 y^{33} + 12y^{34}$\\
$W_{43,5}$&$1+ 42 y^{32} + 18 y^{34} + 3y^{36}$&
$W_{43,6}$&$1+ 15 y^{32} + 48 y^{33}$\\
$W_{43,7}$&$1+ 54 y^{32} + 6 y^{36} + 3y^{40}$&
$W_{43,8}$&$1+ 51 y^{32} + 12 y^{36}$\\
$W_{43,9}$&$1+ 51 y^{32} + 12 y^{36}$&
$W_{43,10}$&$1+ 63 y^{32}$\\
\hline
$W_{48,1}$   &$ 1+ 57y^{36} + 3y^{40} + 3 y^{44}$&
$W_{48,2}$ &$ 1+ 54y^{36} + 9y^{40}$\\
$W_{48,3}$ &$ 1+ 54y^{36} + 9y^{40}$&
$W_{48,4}$   &$ 1+ 45y^{36} +18y^{38}$\\
$W_{48,5}$   &$ 1+ 60y^{36} + 3y^{48}$& &\\
\hline
$W_{52,1}$&$1+ 42y^{39}+15y^{40}+ 6y^{43}$&
$W_{52,2}$&$1+ 45y^{39}+12y^{40}+ 3y^{43}+ 3y^{44}$\\
$W_{52,3}$&$1+ 48y^{39}+12y^{40}+ 3y^{48}$&
$W_{52,4}$&$1+ 45y^{39}+15y^{40}+ 3y^{47}$\\
$W_{52,5}$&$1+ 48y^{39}+ 9y^{40}+ 6y^{44}$ &\\
\hline
$W_{56, 1}$&$1+36y^{42}+ 24y^{43}+ 3y^{48}$&
$W_{56, 2}$&$1+45y^{42}+ 15y^{44}+ 3y^{46}$\\
$W_{56, 3}$&$1+48y^{42}+ 12y^{44}+ 3y^{48}$&
$W_{56, 4}$&$1+42y^{42}+ 21y^{44}$\\
$W_{56, 5}$&$1+36y^{42}+ 21y^{43}+ 3y^{44}+ 3y^{47}$&
$W_{56, 6}$&$1+33y^{42}+ 24y^{43}+ 3y^{44}+ 3y^{46}$\\
\hline
$W_{64,1}$&$1+  51y^{48}+ 12y^{52}$&
$W_{64,2}$&$1+  51y^{48}+ 12y^{52}$\\
$W_{64,3}$&$1+  51y^{48}+ 12y^{52}$&
$W_{64,4}$&$1+  15y^{48}+ 48y^{49}$\\
$W_{64,5}$&$1+  57y^{48}+  3y^{52}+  3y^{60}$&
$W_{64,6}$&$1+  54y^{48}+  6y^{52}+  3y^{56}$\\
$W_{64,7}$&$1+  54y^{48}+  6y^{52}+  3y^{56}$&
$W_{64,8}$&$1+  27y^{48}+ 24y^{49}+ 12y^{50}$\\
$W_{64,9}$&$1+  39y^{48}+ 12y^{49}+ 12y^{51}$&
$W_{64,10}$&$1+  45y^{48}+ 15y^{50}+  3y^{54}$\\
$W_{64,11}$&$1+  42y^{48}+ 18y^{50}+  3y^{52}$&
$W_{64,12}$&$1+  57y^{48}+  6y^{56}$\\
$W_{64,13}$&$1+  63y^{48}$&
$W_{64,14}$&$1+  39y^{48}+ 24y^{50}$\\
$W_{64,15}$&$1+  60y^{48}+  3y^{64}$& & \\
\noalign{\hrule height0.8pt}
\end{tabular}
}
\end{center}
\end{table}

\section{Determination of $d_4(n,3)$}\label{Sec:dim3}

The aim of this section is to establish the following theorem,
which is one of the main results in this paper.

\begin{thm}\label{thm:main}
Suppose that $n \ge 6$.  Then
\[
d_4(n,3)=
\begin{cases}
\lfloor \frac{16n}{21} \rfloor &\text{ if }  
n \equiv 5,9,13,17,18 \pmod{21}, \\
\lfloor \frac{16n}{21} \rfloor-1 &\text{ if } 
 n \equiv 0,1,2,3,4,6,7,8,\\ & \qquad 10,11,12,14,15,16,19,20 \pmod{21}.
\end{cases}
\]
\end{thm}
\begin{rem}\label{rem}
For $n \equiv 2,3,4,8,9,12,13,17,18 \pmod{21}$,
the above result is known~\cite[Table 4]{LLGF}.
\end{rem}

Suppose that there is an (unrestricted)  quaternary $[n,3,d]$ code.
Write $n=21s+t$, where $0 \le t \le 20$.
By the Griesmer bound~\eqref{eq:Gb}, 
we have $d \le \alpha_4(n,3)$, where $\alpha_4(n,3)$ are listed in 
Table~\ref{Tab:Gb} for each $n=21s+t$
($s \ge 1$ if $t=0,1,2$ and $s \ge 0$ if $t=3,4,\ldots,20$).

\begin{table}[thb]
\caption{Griesmer bound $\alpha_4(n,3)$}
\label{Tab:Gb}
\begin{center}
{\small
\begin{tabular}{c|c|c|c|c|c}
\noalign{\hrule height0.8pt}
$n$ & $\alpha_4(n,3)$&
$n$ & $\alpha_4(n,3)$&
$n$ & $\alpha_4(n,3)$\\
\hline
$21s   $&$ 16s    $&$21s+ 7$&$ 16s+ 4 $&$21s+14$&$ 16s+10$ \\
$21s+ 1$&$ 16s    $&$21s+ 8$&$ 16s+ 5 $&$21s+15$&$ 16s+11$ \\
$21s+ 2$&$ 16s    $&$21s+ 9$&$ 16s+ 6 $&$21s+16$&$ 16s+12$ \\
$21s+ 3$&$ 16s+ 1 $&$21s+10$&$ 16s+ 7 $&$21s+17$&$ 16s+12$ \\
$21s+ 4$&$ 16s+ 2 $&$21s+11$&$ 16s+ 8 $&$21s+18$&$ 16s+13$ \\
$21s+ 5$&$ 16s+ 3 $&$21s+12$&$ 16s+ 8 $&$21s+19$&$ 16s+14$ \\
$21s+ 6$&$ 16s+ 4 $&$21s+13$&$ 16s+ 9 $&$21s+20$&$ 16s+15$\\
\noalign{\hrule height0.8pt}
\end{tabular}
}
\end{center}
\end{table}

\begin{lem}[Lu, Li, Guo and Fu~\cite{LLGF}]
\label{lem:LLGF}
If there is a quaternary Hermitian LCD $[n,3,d]$ code, then
there is a quaternary Hermitian LCD $[21s+n,3,16s+d]$ code
for every positive integer $s$.
\end{lem}

By the Griesmer bound~\eqref{eq:Gb} and \cite[Table 4]{LLGF}, we have 
\[
d_4(21s+5,3) = 16s+2 \text{ or } 16s+3. 
\]
The code $C_{26}$ given in Section~\ref{Sec:35} is a quaternary 
Hermitian LCD $[26,3,19]$ code.
By Lemma~\ref{lem:LLGF}, 
there is a  quaternary 
Hermitian LCD  $[21s'+26,3,16s'+19]$ code for a positive integer $s'$.  
Therefore, we have the following:

\begin{prop}\label{prop:1}
For a positive integer $s$, $d_4(21s+5,3) = 16s+3$.
\end{prop}

\begin{rem}\label{rem:Q}
An $[[n,k,d;c]]$ EAQECC $\cC$
encodes $k$ information qubits into $n$ channel qubits
with the help of $c$ pairs of maximally entangled Bell states.
The parameter $d$ is called the minimum weight of $\cC$.
The EAQECC $\cC$
can correct up to $\lfloor \frac{d-1}{2} \rfloor$ 
errors acting on the $n$ channel qubits (see e.g.~\cite{LLG}
and \cite{LLGF}).
An $[[n,k,d;0]]$ EAQECC is a standard quantum code.
An $[[n,k,d;n-k]]$ EAQECC is called {\em maximal entanglement}.
If there is a quaternary Hermitian LCD $[n,k,d]$ code, 
then there is a maximal entanglement
$[[n,k,d;n-k]]$ EAQECC
(see e.g.~\cite{LLG} and \cite{LLGF}).
From Proposition~\ref{prop:1}, 
there is a maximal entanglement 
$[[21s+5,3,16s+3;21s+2]]$ EAQECC for a positive integer $s$.  
It was shown in~\cite[Theorem 7]{LBW} that
$d \le \frac{3n \times 4^k}{4(4^k-1)}$
for an $[[n,k,d;c]]$ EAQECC.
Hence, 
a maximal entanglement $[[21s+5,3,16s+3;21s+2]]$ 
EAQECC meets the above bound
for a positive integer $s$.  
Therefore, the largest minimum weight among all 
maximal entanglement $[[21s+5,3,d;21s+2]]$ EAQECC's is $16s+3$
for a positive integer $s$.  
\end{rem}

From~\cite[Tables 3 and 4]{LLGF} 
and Proposition~\ref{prop:1},
we have 
\[
\begin{array}{ll}
d_4(21s+ 2,3)=16s,
&d_4(21s+ 3,3)=16s+ 1,\\
d_4(21s+ 4,3)=16s+ 2,
&d_4(21s+5,3) = 16s+3,
\end{array}
\]
for a positive integer $s$, and
\[
\begin{array}{ll}
d_4(21s+ 7,3)=16s+ 4,
&d_4(21s+ 8,3)=16s+ 5,\\
d_4(21s+ 9,3)=16s+ 6,
&d_4(21s+12,3)=16s+ 8,\\
d_4(21s+13,3)=16s+ 9,
&d_4(21s+17,3)=16s+12,\\
d_4(21s+18,3)=16s+13,
\end{array}
\]
for a nonnegative integer $s$.
In the remainder of this section, we consider the remaining cases.

As a special case of Theorem~\ref{thm:S-2-9} (ii), we have the following:

\begin{lem}\label{lem:S3}
Suppose that 
\begin{equation*} 
(n_0,d_0) \in 
\left\{\begin{array}{l}
(16,12),
(20,15),
(32,24),
(36,27),
(40,30),
\\
(48,36),
(52,39),
(56,42),
(64,48)
\end{array}\right\}.
\end{equation*}
If there is no  quaternary Hermitian
LCD $[n_0,3,d_0]$ code $C$ with $d(C^{\perp_H}) \ge 2$.
Then
there is no  quaternary Hermitian
LCD $[21s+n_0,3,16s+d_0]$ code $D$ with $d(D^{\perp_H}) \ge 2$
for  a positive integer $s$.
\end{lem}


\begin{prop}\label{prop:21s}
For a positive integer $s$, $d_4(21s,3)=16s-1$.
\end{prop}
\begin{proof}
By Theorem~\ref{thm:S1},  $d_4(21s,3) \le 16s-1$.
There is a  quaternary 
Hermitian LCD $[21,3,15]$ code~\cite[Table 3]{LLGF}.
By Lemma~\ref{lem:LLGF}, 
there is a  quaternary  Hermitian LCD $[21s,3,16s-1]$ code.
\end{proof}

The results in Table~\ref{Tab:35} are used in the following proposition.

\begin{prop}\label{prop:dim3-1}
For a nonnegative integer $s$,
\[
\begin{array}{ll}
d_4(21s+11,3)= 16s+7, & d_4(21s+16,3)= 16s+11, \\
d_4(21s+20,3)= 16s+14. & \\
\end{array}
\]
\end{prop}
\begin{proof}
Suppose that $(n_0,d_0) \in \{(16,12),(20,15),(32,24)\}$
and $s$ is a nonnegative integer.
From Table~\ref{Tab:35}, there is no  quaternary  Hermitian
LCD $[n_0,3,d_0]$ code.
Hence, by Lemma~\ref{lem:S3}, 
there is no  quaternary  Hermitian
LCD $[21s+n_0,3,16s+d_0]$ code $C$ with $d(C^{\perp_H}) \ge 2$.
Suppose that  there is  a  quaternary  Hermitian LCD $[n,3,d]$ code $D$ with 
$d(D^{\perp_H})=1$.
By Lemma~\ref{lem:dd1},
a quaternary 
Hermitian LCD $[n-1,3,d]$ code is constructed.
This contradicts the 
Griesmer bound (see Table~\ref{Tab:Gb}).

It was shown in~\cite[Tables~3 and 4]{LLGF} that
\[
\begin{array}{ll}
d_4(21s+11,3)\ge 16s+7,&
d_4(21s+16,3)\ge 16s+11, \\
d_4(21s+20,3)\ge 16s+14.
\end{array}
\]
From Table~\ref{Tab:35},
it is known that there is no  quaternary 
Hermitian LCD $[11,3,8]$ code.
This completes the proof.
\end{proof}

The results in Proposition~\ref{prop:36-64} 
are used in the following proposition.

\begin{prop}\label{prop:dim3-2}
For a nonnegative integer $s$,
\[
\begin{array}{ll}
d_4(21s+1,3)= 16s-1,&
d_4(21s+ 6,3)= 16s+3,\\
d_4(21s+10,3)= 16s+6,&
d_4(21s+14,3)= 16s+9,\\
d_4(21s+15,3)= 16s+10,&
d_4(21s+19,3)= 16s+13.
\end{array}
\]
\end{prop}
\begin{proof}
Suppose that $s$ is a nonnegative integer and
\[
(n_0,d_0) \in \{
(36,27),
(40,30),
(48,36),
(52,39),
(56,42),
(64,48)
\}.
\]
By Proposition~\ref{prop:36-64} (i), (iii)--(v), 
there is no  quaternary  Hermitian
LCD $[n_0,3,d_0]$ code.
Hence, by Lemma~\ref{lem:S3}, 
there is no  quaternary  Hermitian
LCD $[21s+n_0,3,16s+d_0]$ code $C$ with $d(C^{\perp_H}) \ge 2$.
Now suppose that 
there is a  quaternary  Hermitian 
LCD $[21s+n_0,3,16s+d_0]$ code $D$ with $d(D^{\perp_H}) =1$.
By Lemma~\ref{lem:dd1},
a  quaternary 
Hermitian LCD $[21s+n_0-1,3,16s+d_0-1]$ 
code is constructed.
This contradicts 
$d_4(21s,3)=16s-1$ in Proposition~\ref{prop:21s}
if $21s+n_0 =21s+1$ $(s \ge 1)$,
and this contradicts 
the Griesmer bound (see Table~\ref{Tab:Gb}) otherwise.

It was shown in~\cite[Tables 3 and 4]{LLGF} that
\[
\begin{array}{ll}
d_4(21s+ 1,3) \ge 16s-1, &
d_4(21s+ 6,3)\ge 16s+3, \\
d_4(21s+10,3) \ge 16s+6, &
d_4(21s+14,3) \ge 16s+9, \\
d_4(21s+15,3)\ge 16s+10,&
d_4(21s+19,3)\ge 16s+13.
\end{array}
\]
From Table~\ref{Tab:35} and Proposition~\ref{prop:36-64} (ii),
it is known that there is no  quaternary 
Hermitian LCD $[n,3,d]$ code for 
\begin{multline*}
(n,d)=
( 6, 4),
(10, 7),
(14,10),
(15,11),
\\
(19,14),
(27,20),
(22,16),
(31,23),
(35,26),
(43,32).
\end{multline*}
This completes the proof.
\end{proof}

Combining Propositions~\ref{prop:1}, \ref{prop:21s},
\ref{prop:dim3-1} and \ref{prop:dim3-2} with Remark~\ref{rem},
we determine $d_4(n,3)$ as described in Table~\ref{Tab:B}
and we complete the proof of Theorem~\ref{thm:main}.

\begin{table}[thb]
\caption{$d_4(n,3)$ $(n \ge 6)$}
\label{Tab:B}
\begin{center}
{\small
\begin{tabular}{c|c|c|c|c|c}
\noalign{\hrule height0.8pt}
$n$ & $d_4(n,3)$& Reference &
$n$ & $d_4(n,3)$& Reference \\
\hline
$21s   $& $16s-1   $ & Proposition~\ref{prop:21s} &
  $21s+11$& $16s+ 7$ & Proposition~\ref{prop:dim3-1} \\
$21s+ 1$& $16s-1  $ & Proposition~\ref{prop:dim3-2} &
  $21s+12$& $16s+ 8$ &\cite[Table 4]{LLGF}\\
$21s+ 2$& $16s   $ &\cite[Table 4]{LLGF} &
  $21s+13$& $16s+ 9$ &\cite[Table 4]{LLGF}\\
$21s+ 3$& $16s+ 1$ &\cite[Table 4]{LLGF} &
  $21s+14$& $16s+9 $ & Proposition~\ref{prop:dim3-2} \\
$21s+ 4$& $16s+ 2$ &\cite[Table 4]{LLGF} &
  $21s+15$& $16s+10$ & Proposition~\ref{prop:dim3-2} \\
$21s+ 5$& $16s+ 3$ & Proposition~\ref{prop:1} &
  $21s+16$& $16s+11$ & Proposition~\ref{prop:dim3-1}\\
$21s+ 6$& $16s+ 3$ & Proposition~\ref{prop:dim3-2} &
  $21s+17$& $16s+12$ &\cite[Table 4]{LLGF}\\
$21s+ 7$& $16s+ 4$ &\cite[Table 4]{LLGF} &
  $21s+18$& $16s+13$ &\cite[Table 4]{LLGF}\\
$21s+ 8$& $16s+ 5$ &\cite[Table 4]{LLGF} &
  $21s+19$& $16s+13$ &Proposition~\ref{prop:dim3-2} \\
$21s+ 9$& $16s+ 6$ &\cite[Table 4]{LLGF} &
  $21s+20$& $16s+14$ &Proposition~\ref{prop:dim3-1}\\
$21s+10$& $16s+6 $ & Proposition~\ref{prop:dim3-2} &&\\
\noalign{\hrule height0.8pt}
\end{tabular}
}
\end{center}
\end{table}


\bigskip
\noindent
{\bf Acknowledgment.}
This work was supported by JSPS KAKENHI Grant Number 15H03633.
The authors would like to thank the anonymous referees for 
the useful comments.



\section*{Appendix}
In Appendix, we give a proof of Proposition~\ref{prop:dim2-1}.
For $a=(a_1,a_2,\ldots,a_{5})\in \ZZ_{\ge 0}^{5}$,
we define a quaternary $[n,2]$ code $C(a)$ having generator matrix of the form
$
G(a)=
\left(
\begin{array}{ccccc}
I_2 &  M(a)  \\
\end{array}
\right),
$
where 
\begin{equation}\label{eq:G}
M(a)
=
\left(
\begin{array}{cccccc}
\0_{a_{1}}&\1_{a_{2}}&\1_{a_{3}}&\1_{a_{4}}&\1_{a_{5}}\\
\1_{a_{1}}&\0_{a_{2}}&\1_{a_{3}}&\omega\1_{a_{4}}&\omega^2\1_{a_{5}}\\
\end{array}
\right).
\end{equation}
It is trivial that any quaternary $[n,2]$ code $C$ is equivalent to
some $C(a)$ if $d(C^{\perp_H}) \ge 2$.
By considering all codewords, 
the weight enumerator of the code $C(a)$ is written using
$a_1,a_2,\ldots,a_{5}$ as follows:
\begin{equation*} 
\begin{split}
&
1
+3y^{1+a_1+a_3+a_4+a_5}
+3y^{1+a_2+a_3+a_4+a_5}
\\&
+3y^{2+a_1+a_2+a_4+a_5}
+3y^{2+a_1+a_2+a_3+a_4}
+3y^{2+a_1+a_2+a_3+a_5}.
\end{split}
\end{equation*}
The matrix  $G(a)\overline{G(a)}^T$ is written using
$a_1,a_2,\ldots,a_{5}$ as follows:
\[
\left(
\begin{array}{cc}
1+a_2+a_3+a_4+a_5 & a_3+\omega a_4+\omega^2 a_5 \\
a_3+\omega^2 a_4+\omega a_5 & 1+a_1+a_3+a_4+a_5 
\end{array}
\right).
\]
Hence, the determinant of  $G(a)\overline{G(a)}^T$ is written using
$a_1,a_2,\ldots,a_{5}$ as follows:
\begin{equation}\label{eq:det}
\begin{split}
&
1+a_1+a_2+a_1a_2+a_1a_3+a_1a_4+a_1a_5+a_2a_3+a_2a_4+a_2a_5
\\&
+(\omega+\omega^2)(a_3a_4+a_3a_5+a_4a_5).
\end{split}
\end{equation}

\begin{lem}\label{lem:1}
Suppose that $n \equiv 0,4 \pmod 5$.
If there is a quaternary 
Hermitian LCD $[n,2,\lfloor \frac{4n}{5} \rfloor]$ code $C$, then $d(C^{\perp_H}) \ge 2$.
\end{lem}
\begin{proof}
Suppose that $n \equiv 0,4 \pmod 5$.
Suppose that there is a quaternary 
Hermitian LCD $[n,2,\lfloor \frac{4n}{5} \rfloor]$ code $C$ with 
$d(C^{\perp_H})=1$.
By Lemma~\ref{lem:dd1},
a quaternary Hermitian LCD 
$[n-1,2,\lfloor \frac{4n}{5} \rfloor]$ code is constructed.
This contradicts~\eqref{eq:d}.
\end{proof}

Suppose that $n \equiv 0,4 \pmod 5$.
Let $C$ be a quaternary Hermitian LCD $[n,2,\lfloor \frac{4n}{5} \rfloor]$ code.
By Lemma~\ref{lem:1},
we may assume without loss of generality that
$C=C(a)$, that is,
$C$ has generator matrix of the following form 
$
G(a)=
\left(
\begin{array}{ccccc}
I_2 &  M(a)  \\
\end{array}
\right),
$
where $M(a)$ is listed in~\eqref{eq:G}.
From the length and the minimum weight of $C(a)$, 
$a_1,a_2,\ldots,a_5$ satisfy the following conditions:
\begin{align}
1+\sum_{i \in\{1,2,3,4,5\} \setminus \{j\}}a_i
   &\ge \left\lfloor \frac{4n}{5} \right\rfloor \ (j=1,2), \label{eq:C1}\\
2+\sum_{i \in\{1,2,3,4,5\} \setminus \{j\}}a_i
   &\ge \left\lfloor \frac{4n}{5} \right\rfloor\ (j=3,4,5), \label{eq:C2}\\
2+\sum_{i \in\{1,2,3,4,5\}} a_i&=n. \label{eq:C6}
\end{align}
From~\eqref{eq:C1}--\eqref{eq:C6}, we have
\begin{equation}\label{eq:C7}
\begin{split}
a_i+1 &\le n-\left\lfloor \frac{4n}{5} \right\rfloor \ (i=1,2), \\
a_i     &\le n-\left\lfloor \frac{4n}{5} \right\rfloor \ (i=3,4,5).
\end{split}
\end{equation}

\begin{itemize}
\item Suppose that $n=5s$.  
From~\eqref{eq:C7}, we have
\begin{align*}
a_i \le s-1\ (i=1,2) \text{ and }
a_i \le s\ (i=3,4,5).
\end{align*}
Then we have
\[
n=2+a_1+a_2+a_3+a_4+a_5 \le 5s=n.
\]
Hence, we have
\begin{align*}
a_1=a_2= s-1  \text{ and }
a_3=a_4=a_5= s.
\end{align*}

By~\eqref{eq:det}, using $s$,
the determinant of  $G(a)\overline{G(a)}^T$ is written as
$10s^2-6s$.
Hence, $C(a)$ is not Hermitian LCD for every positive integer $s$.

\item Suppose that $n=5s+4$.
From~\eqref{eq:C7}, we have
\begin{align*}
a_i \le s \ (i=1,2) \text{ and }
a_i \le s+1 \ (i=3,4,5).
\end{align*}
Then we have
\[
n=2+a_1+a_2+a_3+a_4+a_5 \le 5s+5 =n+1.
\]
Hence, we have
\begin{align*}
|\{i \in \{1,2\} \mid a_i\le s-2\}|
=|\{i \in \{3,4,5\} \mid a_i \le s-1\}| &= 0,\\
|\{i \in \{1,2\} \mid a_i=s-1\}|+|\{i \in \{3,4,5\} \mid a_i=s\}| &= 1. 
\end{align*}
This yields that there are the five possibilities for
$a=(a_1,a_2,a_3,a_4,a_5)$, 
where the results are listed in Table~\ref{Tab:dim2-1}.
The determinant $\det$ of  $G(a)\overline{G(a)}^T$ is also
listed in Table~\ref{Tab:dim2-1}.
Therefore, $C_i$ $(i=1,2,\ldots,5)$
is not Hermitian LCD for every positive integer $s$.
\end{itemize}
This completes the proof of Proposition~\ref{prop:dim2-1}.

\begin{table}[thb]
\caption{Case $n=5s+4$}
\label{Tab:dim2-1}
\begin{center}
{\small
\begin{tabular}{c|c|l}
\noalign{\hrule height0.8pt}
$C(a)$ & $a=(a_1,a_2,a_3,a_4,a_5)$ & \multicolumn{1}{c}{$\det$} \\
\hline
$C_1$&
$(s- 1,s,s + 1,s + 1,s + 1)$  &$10s^2 + 10s$\\
$C_2$&
$(s,s - 1,s + 1,s + 1,s + 1)$ &$10s^2 + 10s$\\
$C_3$&
$(s,s,s,s + 1,s + 1)$ &$10s^2 + 10s + 2$\\
$C_4$&
$(s,s,s + 1,s,s + 1)$ &$10s^2 + 10s + 2$\\
$C_5$&
$(s,s,s + 1,s + 1,s)$ &$10s^2 + 10s + 2$\\
\noalign{\hrule height0.8pt}
\end{tabular}
}
\end{center}
\end{table}

\end{document}